\documentclass[12pt]{article}
\usepackage{amsmath,amssymb,amsthm}
\usepackage[pdftex]{color,graphicx}
\usepackage{graphics}
\usepackage{enumerate}
\usepackage{array}
\usepackage[english]{babel}
\usepackage{verbatim,makeidx}
\usepackage{hyperref}
\usepackage{caption}
\usepackage{graphicx,epsfig}
\usepackage{makeidx,verbatim}
\usepackage{graphicx}
\usepackage{mathtools}
\usepackage{tikz-cd}

\usepackage{pdflscape}
\usepackage{rotating}
\usepackage{graphicx}

\usepackage{geometry} 
\usepackage{amsmath}
\usepackage{amsfonts}
\usepackage{amssymb}
\usepackage{amsthm}
\usepackage{mathrsfs}
\usepackage{mathtext}
\usepackage{mathtools}
\usepackage{chngcntr}

\usepackage[all,knot]{xy}

\newcommand{\K}{\mathbb{K}}

\newcommand{\B}{\mathbb{B}}

\newcommand{\Img}[1]{\mathrm{Im}({#1})}
\newcommand{\Ker}[1]{\mathrm{Ker}({#1})}

\newcommand{\ot}{\otimes}

\DeclareMathOperator{\HH}{HH}
\DeclareMathOperator{\HHom}{Hom}

\numberwithin{equation}{section}

\theoremstyle{break}

\newtheorem{rema}[equation]{Remark}
\newtheorem{theo}[equation]{Theorem}
\newtheorem{prop}[equation]{Proposition}

\newtheorem*{coro*}{Corollary}

\title{CUP PRODUCT ON HOCHSCHILD COHOMOLOGY OF A FAMILY OF QUIVER ALGEBRAS}

\author{Tolulope Oke\\}
\date{October 29, 2019}

\begin{document}

\maketitle
\thispagestyle{empty}

\begin{abstract}
Let $k$ be a field, $q\in k$. We derive a cup product formula on the Hochschild cohomology $\text{HH}^*(\Lambda_q)$ of a family $\Lambda_q$ of quiver algebras. Using this formula, we determine a subalgebra of $k[x,y]$ isomorphic to $\text{HH}^*(\Lambda_q)/\mathcal{N}$, where $\mathcal{N}$ is the ideal generated by homogeneous nilpotent elements. We explicitly construct non-nilpotent Hochschild cocycles which cannot be generated by lower homological degree elements, thus disproving the Snashall-Solberg finite generation conjecture.\\
\\
\end{abstract}

\section{Introduction}\label{intro}
The theory of support varieties has been well developed for finite groups using group cohomology. Several efforts were made to develop similar theories for finitely generated modules over finite dimensional algebras using Hochschild cohomology. Hochschild cohomology $\HH^*(\Lambda)$ is graded commutative. If the characteristic  $char(k)\neq 2$, then every homogeneous element of odd degree is nilpotent. Let $\mathcal{N}$ be the set of nilpotent elements of $\HH^*(\Lambda)$, Hochschild cohomology modulo nilpotents 
$\HH^*(\Lambda)/\mathcal{N}$ is therefore a commutative $k$-algebra. For some finite dimensional algebras, it is well known that the Hochschild cohomology ring modulo nilpotents is finitely generated
as an algebra. N. Snashall described many classes of such algebras in section 3 of ~\cite{SVHC}. Before the expository paper ~\cite{SVHC}, it was conjectured in \cite{SVHCR} that Hocschild cohomology modulo nilpotents is always finitely generated
as an algebra for finite dimensional algebras. The first counterexample to this conjecture appeared in ~\cite{XU} where F. Xu used certain techniques in category theory to construct a seven-dimensional category algebra whose Hochschild cohomology ring modulo nilpotents is not finitely generated. There have been since then several constructions of different counter examples to this conjecture\cite{MCSS}. While it is of great use to produce a counterexample, it is equally of importance to discuss the techniques used to produce such examples. Snashall gave a different presentation of the Xu counterexample which we will summarize briefly.

A quiver is a directed graph where loops and multiple arrows (also called
paths) between two vertices are possible. For a field $k$, the path algebra $kQ$, is the $k$-vector space generated by all paths in the quiver $Q$. A vertex is a path of length $0$. By taking multiplication of two paths $x$ and $y$ to be the concatenation $xy$ if the terminal vertex $t(x)$ of $x$ and the origin vertex $o(y)$ of $y$ are equal, and otherwise $0$, $kQ$ becomes an associative ring. Let $I$ be an ideal of $kQ$. The quotient $ \Lambda = kQ/I$ is called a quiver algebra. 


Let $Q$ be the quiver;
\begin{tikzcd}
1 \arrow[out=190,in=270,loop,swap,"b"]
  \arrow[out=90,in=170,loop,swap,"a"]
  \arrow[r,"c"] & 2
\end{tikzcd}\ \\
and let
\begin{equation}\label{quivalg}
\Lambda = \Lambda_q = \frac{kQ}{I}, \qquad I = \langle a^2,b^2, ab-qba, ac \rangle, q\in k
\end{equation}
be a family of quiver algebras. We note the following about $\Lambda_q$ for each $q$.\\
\textbf{Remarks}
\begin{itemize}
\item $\Lambda$ is finitely generated since $Q$ is a finite quiver with finite vertices and arrows. 
 \item $\Lambda$ is a Koszul graded quiver algebra.
\item Let $\Lambda = \oplus_{i\geq 0} \Lambda_i$ be a grading for $\Lambda$, the Koszul dual $\Lambda^!$ of $\Lambda$ is connected to the Yoneda algebra of $\Lambda$ by the following;
\begin{equation}
E(\Lambda) = Ext^*_{\Lambda}(\Lambda_0,\Lambda_0) = \Lambda^! \cong kQ^{opp}/I^{\perp}
\end{equation}
where $Q^{opp}$ is the quiver with opposite arrow,  $I^{\perp} = \langle  a^ob^0+q^{-1}b^0a^0, b^0c^0 \rangle$ and any $v\in KQ$, $v^0$ is
 the correponding arrow in opposite quiver algebra $kQ^{opp}$. Note also that $\Lambda^!$ is generated in degrees 0 and 1.
\item The case where $q=\pm 1$, $I$ belongs to a class of (anti-)commutative ideals studied by Gawell and Xantcha. There is an associated generator graph (of the orthogonal ideal $I^{\perp}$ of $I$) which has no directed cycles. This means that the ideal $I$ is admissible~\cite{GX}.
\end{itemize}

For the case $q=1$ of ~\eqref{quivalg}, the graded center $Z_{gr}(E(\Lambda_q))$ is given by the following 
 $$Z_{gr}(E(\Lambda_1)) = \begin{cases} k\oplus k[a,b]b, &\text{ if } char(k) = 2 \\
k\oplus k[a^2,b^2]b^2, &\text{ if } char(k)\neq 2 \end{cases}$$
where the degree of $b$ is 1, and that of $ab$ is 2.

We now present a theorem of Snashall\rq{}s with respect to the finite generation conjecture.
\begin{theo}\label{snalsol}
Let $k$ be a field and $\Lambda_1$ be a member of the class of quiver algebras given in ~\eqref{quivalg}, and $\mathcal{N}$ be the set of nilpotent elements of $\HH^*(\Lambda_1),$ then
$$\HH^*(\Lambda_1)/\mathcal{N}\cong Z_{gr}(E(\Lambda_1)) = \begin{cases} k\oplus k[a,b]b, &\text{ if } char(k) = 2 \\
k\oplus k[a^2,b^2]b^2, &\text{ if } char(k)\neq 2 \end{cases}$$
where the degree of $b$ is 1, and that of $ab$ is 2.
\end{theo}\ \\
\textbf{Our Result:} In this paper, we study the Hochschild cohomology ring of the family $\Lambda_q$ of quiver algebras of equation~\eqref{quivalg}. We give a formula for the dimension of the space of Hochschild cocycles $\Ker{d^*}$ where $d^*: \HHom_{\Lambda^e}(\K_*,\Lambda) \rightarrow \HHom_{\Lambda^e}(\K_*,\Lambda) $. We show that this number increases as the homological dimension grows. However, with a generalized cup product formula, we show that many of these cocycles are nilpotent with respect to the cup product and all cocycles of odd homological degrees are nilpotent.

While some authors have presented the result of Theorem ~\ref{snalsol} using the graded center of the Yoneda algebra $E(\Lambda)$, we give the same result using a different approach. We present our result in Theorem~\ref{tolwit} without using the isomorphism $\HH^*(\Lambda_1)/\mathcal{N}\cong Z_{gr}(E(\Lambda_1))$ of Theorem~\eqref{snalsol} ( or the  map $\phi_M$ of equation~\eqref{phim} defined in the next section). This approach involves a direct computation and the use of the generalized cup product formula on elements of Hochschild cohomology in Propostion~\eqref{prop0}. Furthermore, we give explicit presentation of the elements $a^{2n-r}b^r$ in $k[a,b]$ of Theorem~\ref{snalsol} as cocycles in Hochschild cohomology. We noted that these elements cannot be generated from other elements of lower homological degrees. Our main results are the following;
\begin{prop}
Let $\phi: \K_m\rightarrow \Lambda$, and $\mu: \K_n\rightarrow \Lambda$, be two Hochschild cocycles, and $\{\varepsilon^m_r\}_{r=0}^{t_m}$  are basis elements of $\K_m$ such that for each $0 \leq r\leq t_m$, $\phi(\varepsilon^m_r)=\phi^m_r$. Then the following gives
a formula for the cup product on Hochschild cohomology.\\
$(\phi\smallsmile\mu)(\varepsilon^{m+n}_k) = (\phi\mu)^{m+n}_k = \begin{cases} (-1)^{mn}\phi^m_0\mu^n_0, & when\; k=0\\
(-1)^{mn} T^{m+n}_k & when\; 0<k<m+n\\
 (-1)^{mn}\phi^m_m\mu^n_n, & when\; k=m+n\\
 (-1)^{mn}\phi^m_0\mu^n_{n+1}, & when\;k=m+n+1
\end{cases}$
$$T^{m+n}_k = \sum_{j=max\{0,k-n\}}^{min\{m,k\}} (-q)^{j(n-k+j)}\phi^m_j\mu^n_{k-j}, \quad\qquad 0<k<m+n.$$
\end{prop}
\begin{theo}
Let $k$ ($char(k)\neq 2$) be a field and $\Lambda_q = \frac{kQ}{I}$ be the family of quiver algebras of ~\eqref{quivalg}, and $\mathcal{N}$ the set of nilpotent elements of $\HH^*(\Lambda_q),$  then
$$\HH^*(\Lambda_q)/\mathcal{N} = \begin{cases} \HH^0(\Lambda_q)/\mathcal{N} \cong Z(\Lambda_q) ^*\cong k,  &\text{ if } q\neq \pm 1 \\
Z(\Lambda_q)^* \oplus k[x^2,y^2]y^2 \cong k \oplus k[x^2,y^2]y^2 , &\text{ if } q =\pm 1 \end{cases}$$
where the degree of $y$ is 1, and that of $xy$ is 2.
\end{theo}

\section{Preliminary}\label{prelim}
A $k$-algebra $\Lambda$, is a graded quiver algebra if and only if there exists a finite quiver $Q$ and a homogeneous admissible ideal $I\subseteq kQ$ for which $\Lambda \cong kQ/I ~\cite{MV}.$ Let $(kQ)_n$ be the vector subspace of $kQ$ containing paths and linear combination of paths of length of $n$. Let $I\subset (kQ)_2$, then $\Lambda \cong kQ/I$ is a quadratic quiver algebra.
We define the quadratic dual $\Lambda^!$ of $\Lambda$ to be $\Lambda^! \cong kQ^{opp}/ I^{\perp},$
where $Q^{opp}$ is the quiver $Q$ with opposite arrows. Since $kQ^{opp}$ is also a $k$-vector space, it has a dual basis with respect to $kQ$. We define for each basis element $v_i\in (kQ^{opp})_2,$ a corresponding $x_i\in (kQ)_2$  such that  the bilinear form 
$\langle v_i,x_j\rangle = \begin{cases} 1, &\text{if }i=j \\ 0 & \text{ if } i\neq j. \end{cases}$ 
$$\text{Let} \quad  I^{\perp} = \{ v_i\in(kQ^{opp})_2 | \langle v_i,x_j\rangle = 0,\; \forall\; x_j\in I\subset (kQ)_2 \}.$$

If $\Lambda$ is a graded quiver algebra, then $\Lambda\cong \bigoplus_{i=0} \Lambda_i$. We denote by $\mathfrak{r} = \bigoplus_{i>0} \Lambda_i,$ the Jacobson radical of $\Lambda$ and 
$\Lambda_0 = \Lambda/\mathfrak{r}$ is isomorphic to a direct sum of a finite number of copies of $k$. The Yoneda algebra $E(\Lambda)$ of $\Lambda$ is given by $E(\Lambda) = Ext^*_{\Lambda}(\Lambda_0,\Lambda_0)$. For Koszul algebras $\Lambda= kQ/I$, $(kQ)_2\supseteq I$  is quadratic and generated by minimal uniform relations of homogeneous elements of degree 2, $E(\Lambda) \cong \Lambda^!$ and  $E(\Lambda)$ is generated in degrees 0 and 1~\cite{MV}.

For any two left $\Lambda$-modules $M$ and $N$, the Hochschild cohomology ring $\HH^*(\Lambda)$ acts on $Ext^*_{\Lambda}(M,N)$ in such a way that $Ext^*_{\Lambda}(M,N)$ is an $\HH^*(\Lambda)$-module. That is, there is a map 
\begin{equation}\label{equ1}
 \HH^*(\Lambda)\times Ext^*_{\Lambda}(M,N) \longrightarrow Ext^*_{\Lambda}(M,N)
\end{equation}
that defines a right (also left) module action of $\HH^*(\Lambda)$ on $Ext^*_{\Lambda}(M,N)$ ~\cite{HCSW}.
A consequence of the module action of equation~\eqref{equ1} is the following Proposition given in ~\cite{SVHCR}.

\begin{prop}
Let $M$ be a $\Lambda-$module. The map
\begin{equation}\label{phim}
\phi_M:\HH^*(\Lambda)\longrightarrow Ext^*_{\Lambda}(M,M)
\end{equation}
defined at the chain level by $\phi_M(f) = f\otimes 1_M$ is a ring homomorphism whose image is contained in the graded center
$Z_{gr}( Ext^*_{\Lambda}(M,M))$. 
\end{prop}

For Koszul algebras, the image of the map $\phi_M$ was shown to be equal to the graded center $Z_{gr}( Ext^*_{\Lambda}(M,M))$, where $M\cong \Lambda/\mathfrak{r}$ ~\cite{MSKA}. The graded center is the set of all homogeneous $z\in Ext^*_{\Lambda}(M,M)$ for which
$z\alpha = (-1)^{|z||\alpha|}\alpha z$, for all homogeneous $\alpha$.

For a $k$-algebra $\Lambda$, denote by $\Lambda^e:=\Lambda\ot\Lambda^{op}$ its enveloping algebra having the tensor product algebra structure. $\Lambda^{op}$ is $\Lambda$ with the opposite multiplication. We have then that a left $\Lambda^e$-module is a $\Lambda$-bimodule and vice versa. In ~\cite{MSKA}, the following ideas were used to define a minimal projective resolution of $\Lambda$ as a $\Lambda^e$-module.

For a finite quiver $Q$, let $R=kQ$, and $I\subset kQ$ an ideal of $R$. It was shown in ~\cite{ROKA} that there exist integers $t_n$ and uniform elements $\{f^n_i\}_{i=0}^{t_n}$ in $R$ such that for $n\geq 0$, there is a filtration of right ideals 
$$\cdots \subseteq \bigoplus_{i=0}^{t_n} f^n_iR \subseteq \bigoplus_{i=0}^{t_{n-1}} f^{n-1}_iR\subseteq \cdots \subseteq\bigoplus_{i=0}^{t_1} f^1_iR\subseteq \bigoplus_{i=0}^{t_0} f^0_iR = R$$
in $R$. This filtration was then employed to construct a minimal projective resolution $\K_{\bullet}$ of $\Lambda$ which will be used in defining Hochschild cohomology. For the family of quiver algebras in~\eqref{quivalg}, take $\Lambda_0$ to be the subalgebra of $\Lambda$ generated by the vertices. We immediately view $R$ as the tensor algebra $T_{\Lambda_0}(\Lambda_1)$ where $\Lambda_1$ is the $\Lambda_0$-bimodule generated by the arrows in $Q$. Since $\Lambda$ is Koszul, each $f^n_i$ can be viewed as a linear combination of paths of length $n$. Hence it can be viewed as an element in $\Lambda_1^{\ot_{\Lambda_0}n}$ for all $n$. This makes it possible to embed them into the bar resolution. We choose  $\{f^0_i\}_{i=0}^{t_0}$ to be the set of vertices in $Q$,  $\{f^1_i\}_{i=0}^{t_1}$ to be the set of arrows in $Q$ while $\{f^2_i\}_{i=0}^{t_2}$ to be a minimal set of homogeneous generators of degree two for $I$. We use the notation $\Gamma^n$ to denote the set containing all linear combinations of homogeneous elements of degree $n$, viewed as elements of $\Lambda_1^{\ot_{\Lambda_0}n}.$\\
In summary, by taking $e_i \;(i=0,1)$ to be the arrow of length $0$ (also idempotents) associated with the vertex $i$, we have\
\begin{align*} \Gamma^0 &= \{e_1,e_2\} = \{f_0^0,f_1^0\} ,  \; t_0 = |\Gamma^0| ,\\
\Gamma^1 &= \{a,b,c\} = \{f_0^1,f_1^1,f_2^1\} , \; t_1 = |\Gamma^1|,\\
\Gamma^2 &= \{a\otimes a, a\ot b - qb\otimes a, b\otimes b, a\otimes c \} \\
&= \{f^1_0\ot f^1_0, f^1_0\ot f^1_1 - qf^1_1\ot f^1_0, f^1_1\ot f^1_1, f^1_0\ot f^1_2 \} = \{f_0^2,f_1^2,f_2^2,f^2_3\} ,\; t_2 = |\Gamma^2|.
\end{align*}
and for  $ n\geq 2$ 
\begin{equation}\label{gammas}
\Gamma^n = \bigg\{f^n_i = \begin{cases} 
 a^{\otimes n}, &\text{ when } i=0 \\
 f^{n-1}_{s-1}\otimes b + (-q)^s f^{n-1}_s\otimes a ,&\text{ when } (0<i<n) \\
 b^{\otimes n}, &\text{ when } i=n\\
 a^{\otimes(n-1)}\otimes c,&\text{ when } i=n+1
\end{cases}\bigg\} 
\end{equation}
and set $t_n = |\Gamma^n|$. The minimal projective resolution $(\K, d)$ of $\Lambda$ over $\Lambda^e$-modules is given by 
\begin{equation}
\mathbb{K}_n = \bigoplus_{i=0}^{t_n}\Lambda o(f_i^n)\otimes_k t(f_i^n)\Lambda
\end{equation}
To describe the differentials $d$, we need to define the basis elements of $\K_n$ for each $n$. Let $\varepsilon^n_i = (0,\ldots,0,o(f^n_i)\otimes_k t(f^n_i),0,\ldots,0)\in\mathbb{K}_n$ where the element $o(f^n_i)\otimes_k t(f^n_i)$ is in the $i$-th position counting from 0. We now define the differentials on the resolution
\begin{equation}
\cdots \stackrel{d_{n+1}}{\rightarrow}\K_n\stackrel{d_n}{\rightarrow}\K_{n-1}\stackrel{d_{n-1}}{\rightarrow}\cdots \stackrel{d_2}{\rightarrow}\K_1\stackrel{d_1}{\rightarrow}\K_0\stackrel{\pi}{\rightarrow}\Lambda \rightarrow 0
\end{equation}
as follows
\begin{align*}
d_1(\varepsilon^1_2) &= c\varepsilon^0_1 - \varepsilon^0_0 c\\
d_n(\varepsilon^n_i) &= (1-\partial_{n,i})[a\varepsilon^{n-1}_i +(-1)^{n-i}q^i\varepsilon^{n-1}_i a] \\
&+(1-\partial_{i,0})[(-q)^{n-i}b\varepsilon^{n-1}_{i-1}+(-1)^{n}\varepsilon^{n-1}_{i-1} b], \;\;\text{for}\;\;i\leq n\\
d_n(\varepsilon^n_{n+1}) &= a\varepsilon^{n-1}_n + (-1)^n\varepsilon^{n-1}_0 c, \;\;\text{when}\;\;n\geq 2
\end{align*}
where $\partial_{i,j} = \begin{cases} 1 & \text{ if } i=j\\ 0, &\text{ otherwise } \end{cases}$. Let $\B = \B_{\bullet}(\Lambda)$ denote the bar resolution of $\Lambda$ given by;
\begin{equation}
\B_{\bullet}: = \qquad \qquad\cdots \rightarrow \Lambda^{\ot (n+2)}\stackrel{\delta_n}{\rightarrow}\Lambda^{\ot (n+1)}\stackrel{\delta_{n-1}}{\rightarrow} \cdots \stackrel{\delta_2}{\rightarrow}\Lambda^{\ot 3}\stackrel{\delta_1}{\rightarrow}\Lambda^{\ot 2}\stackrel{\pi}{\rightarrow}\Lambda \rightarrow 0
\end{equation}
The differentials $\delta_i$ are given by
\begin{equation}
\delta_n(a_0\otimes a_1\otimes\cdots\otimes a_{n+1}) = \sum_{i=0}^n (-1)^i a_0\otimes\cdots\otimes a_ia_{i+1}\otimes\cdots\otimes a_{n+1}
\end{equation}
for each element $a_i\in \Lambda \;(0\leq i\leq n+1)$ and $\pi$, the multplication map. The Hochschild cohomology of $\Lambda$ is defined to be
$$\HH^*(\Lambda) = Ext^*_{\Lambda^e}(\Lambda,\Lambda) = \bigoplus_{n\geq 0} H^n(\HHom_{\Lambda^e}(\B_{\bullet}(\Lambda),\Lambda))$$
There is a natural embedding of $\K$ into $\B$. That is, there is a map
$$\iota_{\bullet}: \K_{\bullet}\longrightarrow\B_{\bullet}$$
lifting the identity on $\Lambda$ such that for each $n$, $\iota_n(\varepsilon^n_i) = 1\otimes f^n_i\otimes 1$. We recall that $\mathbb{B}_n = \Lambda^{\otimes (n+2)}$, and so, the following diagram is commutative  that is  $\iota_{n-1} d_{n} (\varepsilon^n_i) = \delta_n \iota_n (\varepsilon^n_i)$ (see Proposition 2.1 of~\cite{MSKA} for a proof).
\begin{equation}\label{comm diag}
\xymatrix{\K_{\bullet} : \cdots \ar[r] & \K_2 \ar@{-->}[d]_-{\iota_2} \,\,\ar[r]^-{d_2} & \,\,\, \K_1 \ar@{-->}[d]_-{\iota_1} \ar[r]^-{d_1} 
& \K_0\ar[d]_-{\iota_0}  \ar[r]^-{\pi} & \Lambda \ar[d]_-{1_{\Lambda}} \ar[r] & 0\\
\B_{\bullet} : \cdots \ar[r] & \B_2  \ar[r]^-{\delta_2} & \B_1 \,\,\, \ar[r]^-{\delta_1} & \,\,\B_0 \ar[r]^-{\pi} & \Lambda \ar[r] & 0 }
\end{equation}\ \\
We apply the functor $\HHom_{\Lambda^e}(-,\Lambda)$ to the projective resolution $\K$ of $\Lambda^e$- modules. Defining  $\mathbb{\hat{K}}_i = \HHom_{\Lambda^e}(\K_i,\Lambda),$ we obtain the following complex\\
\\
$0 \longrightarrow \mathbb{\hat{K}}_0 \stackrel{d_1^*}\longrightarrow\mathbb{\hat{K}}_1 \stackrel{d_2^*}\longrightarrow \mathbb{\hat{K}}_2 \stackrel{d_3^*}\longrightarrow \mathbb{\hat{K}}_3 \stackrel{d_4^*}\longrightarrow \mathbb{\hat{K}}_4 \stackrel{d_5^*}\longrightarrow \mathbb{\hat{K}}_5
\stackrel{d_6^*}\longrightarrow \mathbb{\hat{K}}_6 \longrightarrow \cdots$\\
\\
We note that $t_0=1$, therefore,
\begin{align*}
\mathbb{\hat{K}}_0 = \HHom_{\Lambda^e}(\mathbb{K}_0,\Lambda) &= \HHom_{\Lambda^e}(\Lambda o(f^0_0)\otimes_k t(f^0_0)\Lambda \oplus \Lambda o(f^0_1)\otimes_k t(f^0_1)\Lambda,\Lambda)\\
&= \HHom_{\Lambda^e}(\Lambda o(f^0_0)\otimes_k t(f^0_0)\Lambda,\Lambda)\oplus \HHom_{\Lambda^e}(\Lambda o(f^0_1)\otimes_k t(f^0_1)\Lambda,\Lambda)\\
&= o(f^0_0)\Lambda t(f^0_0)\oplus o(f^0_1)\Lambda t(f^0_1)
\end{align*}
For a fixed $n$ and $i$, let $\phi\in \HHom_{\Lambda^e}(\Lambda o(f^n_i)\otimes_k t(f^n_i)\Lambda,\Lambda)$, suppose that $\phi(0,\cdots,0,o(f^n_i)\otimes_k t(f^n_i),0,\cdots,0) = \phi(\varepsilon^n_i) =\lambda^n_i\in \Lambda$, then
\begin{align*}
o(f^n_i)\lambda^n_i t(f^n_i) &= o(f^n_i)\phi(o(f^n_i)\otimes_k t(f^n_i)) t(f^n_i)\\
 &= \phi(o(f^n_i)^2\otimes_k t(f^n_i)^2) = \phi(o(f^n_i)\otimes_k t(f^n_i)) = \lambda^n_i 
 \end{align*}
where each $o(f^n_i), t(f^n_i)$ is any of the idempotents $e_1$ or $e_2$. We have the following isomorphisms of $\Lambda^e$-modules: $\HHom_{\Lambda^e}(\Lambda o(f^n_i)\otimes_k t(f^n_i)\Lambda,\Lambda) \simeq o(f^n_i)\Lambda \; t(f^n_i).$
So 
\begin{equation}
 \hat{\mathbb{K}}_n = \HHom_{\Lambda^e}(\mathbb{K}_n,\Lambda)= \bigoplus_{i=0}^{t_n}o(f^n_i)\Lambda \; t(f^n_i)
 \end{equation}
In general when $\phi\in \hat{\mathbb{K}_n}$ such that $\phi(\varepsilon^n_i) = \lambda^n_i$ for $0\leq i\leq t_n$, we simply write 
$$\phi = (\lambda^n_0,\cdots,\lambda^n_i,\cdots, \lambda^n_{t_n}).$$

\section{Proof of Main Results}\label{results}
In this section, we give a sequence of propositions leading to the main results. We start by finding elements of Hochschild 0-cocycles, which in theory corresponds to elements in the center of the algebra.\\
\\
\textbf{The 0th Hochschild cohomology  ($\HH^0(\Lambda) = \frac{\ker d^*_1}{\Img 0}$).}

Let $\phi\in \ker d^*_1\subseteq \hat{\mathbb{K}_0} = \HHom_{\Lambda^e}(\mathbb{K}_0,\Lambda)$,  such that $\phi = (\lambda^0_0,\lambda^0_1)$, for some $\lambda^0_1,\lambda^0_1\in \Lambda.$ We solve for the $\lambda^0_i$ $(i=0,1)$ for which $d^*_1\phi(\varepsilon^1_i) = 0$ as follows
   \begin{align*}
   d^*_1\phi(\varepsilon^1_0) = \phi d_1(\varepsilon^1_0) &= \phi(a(\varepsilon^0_0) +(-1)^1q^0 (\varepsilon^0_0)a) \\
   &= a\lambda^0_0  - \lambda^0_0 a= 0\\
   d^*_1\phi(\varepsilon^1_1) = \phi d_1(\varepsilon^1_1) &= \phi((-q)^0b(\varepsilon^0_0) - (\varepsilon^0_0)b) \\
   &= b\lambda^0_0  - \lambda^0_0 b= 0\\
   d^*_1\phi(\varepsilon^1_2) = \phi d_1(\varepsilon^1_2) &= \phi(c(\varepsilon^0_1) - (\varepsilon^0_0)c) \\
   &= c\lambda^0_1  - \lambda^0_0 c= 0
\end{align*}
\textbf{If $q=1$}, then $ab-ba=0,$ we get the following set of solutions: $\phi = (a,0)$, $ (ab,0)$,$(0,a)$, $(0,b)$, $(e_1,e_2)$ or $ (0,e_1).$ By identifying each solution $( \lambda^0_0 ,  \lambda^0_1) $ with  $ (o(f^0_0) \lambda^0_0 t(f^0_0), o(f^0_1) \lambda^0_1 t(f^0_1)) = ( e_1 \lambda^0_0 e_1, e_2 \lambda^0_1 e_2) $,  we need to have $o(\lambda^0_0) = t(\lambda^0_0)=e_1$ and  $o(\lambda^0_1) = t(\lambda^0_1)=e_2,$ we eliminate some solutions to have the following unique set of solutions $\phi_1 =  (a , 0), \phi_2= (ab , 0)$  and $\phi_3 = (e_1, e_2) $ \\
    \\
\textbf{If $q=- 1$}, then $ab+ba=0, $ we get the same unique set of solutions: $\phi_1 = (a , 0)$ (if $char(k)=2),  \phi_2= (ab , 0)$ and  $\phi_3 = (e_1, e_2). $\\
\\
\textbf{If $q\neq 1$}, then $ab-qba=0,$ we get $\phi_2 = (ab , 0)$  and $\phi_3 = (e_1, e_2). $ Therefore, the $\Lambda^e$-module homomorphisms $\phi_1, \phi_2, \phi_3$ form a basis for the kernel of $d^*_1$ as a $k$-vector space. That is, 
$$\ker d^*_1 = span_k\{ \phi_1,\phi_2,\phi_3 \}.$$
In summary we obtain for any $q\in k$ that,
$$ \HH^0(\Lambda) = \frac{\ker d^*_1}{\Img 0} = \begin{cases}
span_k\{ (a,0),(ab,0),(e_1,e_2)\}, &\text{ if  } q=1 \\
span_k\{ (ab,0),(e_1,e_2)\}, &\text{ if  } q=-1 \\
span_k\{ (a,0),(ab,0),(e_1,e_2)\}, &\text{ if  } q=-1 \text{ and } char(k)=2\\
span_k\{ (ab,0),(e_1,e_2)\}, &\text{ for every other  } q\neq\pm1\\
\end{cases}$$
where each ordered pairs are in $e_1\Lambda e_1\oplus e_2\Lambda e_2$.\\
\begin{rema}\label{bigrem}
We note that each Hochschild 0-cocycle of the set $\{(a,0),(ab,0)\}$ corresponds to an element in the set $\{a,ab\}$ of the center of the algebra $\Lambda_q$. As we will see later in Remark ~\ref{bigrem1}, these elements are nilpotent with respect to the cup product but the 0-cocycle $\phi_3 = (e_1,e_2)$ is not, since $e_1,e_2$ are idempotent elements. We then identify $span_k\{ (e_1,e_2)\}$ to be the subalgebra of $\HH^*(\Lambda_q)$ isomorphic to $k$ because $e_1+e_2 = 1_{\Lambda_q}$. This brings us to make the following deduction for any $q\in k$
\begin{equation}\label{center}
 \HH^0(\Lambda)/\mathcal{N} = \frac{\ker d^*_1}{\Img 0} = span_k\{ (e_1,e_2)\}\cong k.
 \end{equation}
\end{rema}\ \\
\textbf{Higher Hochschild cocycles }

We now give the following counting proposition about the dimension of the kernels of the differentials  $d^{*}_{n+1}:\hat{\K}_n\rightarrow \hat{\K}_{n+1}$.
\begin{prop}\label{prop}
Let $k$ be a field and let $\Lambda_q = \frac{kQ}{I}$ where $Q$ is the quiver of equation~\eqref{quivalg}. Hochschild cohomology group is given by $HH^n(\Lambda) = \frac{ker d^*_{n+1}}{Im\;d^*_n}$, and for $n\neq 0$
\begin{center}
$\dim (ker\;d^*_{n+1}) = \begin{cases} 
2(n+2),\;\textit{when n is odd} \\
\frac{5n}{2}+4,\;\textit{when n is even} 
\end{cases}$  \qquad $q =  1$ \\[4pt]
$\dim (ker\;d^*_{n+1})  = \begin{cases} 
\frac{5n}{2}+4,\;\textit{when n is odd}\\
2(n+2),\;\textit{when n is even} 
\end{cases}$  \qquad $q = -1$ \\[4pt]
$\dim (ker\;d^*_{n+1})  = n+2,\;\textit{for any integer n }$  \qquad $q \neq \pm 1$ 
\end{center} 
as a $k$-vector space.
\end{prop}
\begin{proof}
Let $\phi\in ker\;d^*_{n+1}$, with $\phi = (\phi^n_0,\phi^n_1,\cdots,\phi^n_n,\phi^n_{n+1})$. The elements  $\phi^n_i = \phi(\varepsilon^n_i), i=0,\cdots,n+1$ are obtained by solving the following set of equations

\textit{\textbf{For any $n$ or $q$ }}
\begin{align*}
d^*_{n+1}\phi(\varepsilon^{n+1}_0) &= a\phi(\varepsilon^n_0) + (-1)^{n+1}\phi(\varepsilon^n_0)a\\
&= a\phi^n_0 \pm \phi^n_0a = 0 \qquad \text{and}\\
d^*_{n+1}\phi(\varepsilon^{n+1}_{n+2}) &= a\phi(\varepsilon^n_{n+1}) + (-1)^{n+1}\phi(\varepsilon^n_0)c\\
&= a\phi^n_{n+1} \pm \phi^n_0c = 0
\end{align*}
for these set of equations to be zero, we should have $\phi^n_0\in span_k\{a,c,ab,bc\}$ and $\phi^n_{n+1}\in span_k\{a,c,ab,bc\}$. But we recall that $\phi^n_0\in e_1\Lambda e_1$, and $\phi^n_{n+1}\in e_1\Lambda e_2$, we thus obtain the following $\phi^n_0\in span_k\{a,ab\}$ and $\phi^n_{n+1}\in span_k\{c, bc\}.$ The rest of this proof involves solving the general set of equations;
\begin{align*}
d^*_{n+1}\phi(\varepsilon^{n+1}_r) &= a\phi(\varepsilon^n_r) + (-1)^{n+1-r}q^r\phi(\varepsilon^n_r)a + (-q)^{n+1-r}b\phi(\varepsilon^n_{r-1}) + (-1)^{n+1}\phi(\varepsilon^n_{r-1})b\\
&= a\phi^n_r + (-1)^{n+1-r}q^r\phi^n_r a + (-q)^{n+1-r} b\phi^n_{r-1} + (-1)^{n+1} \phi^n_{r-1} b = 0\\
d^*_{n+1}\phi(\varepsilon^{n+1}_{r+1}) &= a\phi^n_{r+1} + (-1)^{n-r}q^{r+1}\phi^n_{r+1} a + (-q)^{n-r} b\phi^n_{r} + (-1)^{n+1} \phi^n_{r} b = 0
\end{align*}
for different values of $n,r$ and $q$. We recall that $q=\pm1$ implies $ab  \mp ab=0$

\textbf{\textit{When $n$ is even, $r$ is even, $q=1$}}\\
we obtain $\phi^n_r$ by setting $\phi^n_{r-1} = \phi^n_{r+1} = 0$ solving
\begin{align*}
d^*_{n+1}\phi(\varepsilon^{n+1}_r) &= a\phi^n_r - \phi^n_r a = 0 \;\;\text{ and }\\
d^*_{n+1}\phi(\varepsilon^{n+1}_{r+1}) &=  b\phi^n_{r} - \phi^n_{r} b=0.
\end{align*}
We can only obtain both equations equal to 0 if $\phi^n_{r}\in span_k\{a,b,ab,bc, e_1\}$. Again we recall that $\phi^{n}_r\in e_1\Lambda e_1$, so $\phi^n_{r}\in span_k\{a,b,ab, e_1\}$.

\textbf{\textit{When $n$ is even, $r$ is odd, $q=1$}}\\
we obtain $\phi^n_r$ by setting $\phi^n_{r-1} = \phi^n_{r+1} = 0$ solving
\begin{align*}
d^*_{n+1}\phi(\varepsilon^{n+1}_r) &= a\phi^n_r + \phi^n_r a = 0 \;\;\text{ and }\\
d^*_{n+1}\phi(\varepsilon^{n+1}_{r+1}) &=  - b\phi^n_{r} - \phi^n_{r} b=0
\end{align*}
We can only obtain both equations equal to 0 if $\phi^n_{r}\in span_k\{ab,bc\}$. Again $\phi^{n}_r\in e_1\Lambda e_1,$ so $\phi^n_{r}\in span_k\{ab\}$.

\textbf{\textit{When $n$ is odd, $r$ is even, $q=1$}}\\
we obtain $\phi^n_r$ by setting $\phi^n_{r-1} = \phi^n_{r+1} = 0$ solving
\begin{align*}
d^*_{n+1}\phi(\varepsilon^{n+1}_r) &= a\phi^n_r + \phi^n_r a = 0 \;\;\text{ and }\\
d^*_{n+1}\phi(\varepsilon^{n+1}_{r+1}) &=  - b\phi^n_{r} + \phi^n_{r} b=0
\end{align*}
We can only obtain both equations equal to 0 if $\phi^n_{r}\in span_k\{a,ab,bc\}$ and finally we get $\phi^n_{r}\in span_k\{a,ab\}$.

\textbf{\textit{When $n$ is odd, $r$ is odd, $q=1,$}}\\
we obtain $\phi^n_r$ by setting $\phi^n_{r-1} = \phi^n_{r+1} = 0$ solving
\begin{align*}
d^*_{n+1}\phi(\varepsilon^{n+1}_r) &= a\phi^n_r - \phi^n_r a = 0 \;\;\text{ and }\\
d^*_{n+1}\phi(\varepsilon^{n+1}_{r+1}) &=   b\phi^n_{r} + \phi^n_{r} b=0
\end{align*}
we can only obtain both equations equal to 0 if $\phi^n_{r}\in span_k\{ab,bc, b\}$. Like before we obtain $\phi^n_{r}\in span_k\{b, ab\}.$

\textbf{\textit{When $n$ is even, $r$ is even, $q=-1$}}\\
we obtain $\phi^n_r$ by setting $\phi^n_{r-1} = \phi^n_{r+1} = 0$ solving
\begin{align*}
d^*_{n+1}\phi(\varepsilon^{n+1}_r) &= a\phi^n_r - \phi^n_r a = 0 \;\;\text{ and }\\
d^*_{n+1}\phi(\varepsilon^{n+1}_{r+1}) &=  b\phi^n_{r} - \phi^n_{r} b=0
\end{align*}
We have the solution $\phi^n_{r}\in span_k\{ab, e_1\}$.

\textbf{\textit{When $n$ is even, $r$ is odd, $q=-1$}}\\
we obtain $\phi^n_r$ by setting $\phi^n_{r-1} = \phi^n_{r+1} = 0$ solving
\begin{align*}
d^*_{n+1}\phi(\varepsilon^{n+1}_r) &= a\phi^n_r + \phi^n_r a = 0 \;\;\text{ and }\\
d^*_{n+1}\phi(\varepsilon^{n+1}_{r+1}) &=   b\phi^n_{r} - \phi^n_{r} b=0
\end{align*}
We can only obtain both equations equal to 0 if $\phi^n_{r}\in span_k\{ab,b\}$.

\textbf{\textit{When $n$ is odd, $r$ is even, $q=-1$}}\\
we obtain $\phi^n_r$ by setting $\phi^n_{r-1} = \phi^n_{r+1} = 0$ solving
\begin{align*}
d^*_{n+1}\phi(\varepsilon^{n+1}_r) &= a\phi^n_r + \phi^n_r a = 0 \;\;\text{ and }\\
d^*_{n+1}\phi(\varepsilon^{n+1}_{r+1}) &=   b\phi^n_{r} + \phi^n_{r} b=0
\end{align*}
We can only obtain both equations equal to 0 if $\phi^n_{r}\in span_k\{a,b,ab\}$.

\textbf{\textit{When $n$ is odd, $r$ is odd, $q=-1$}}\\
we obtain $\phi^n_r$ by setting $\phi^n_{r-1} = \phi^n_{r+1} = 0$ solving
\begin{align*}
d^*_{n+1}\phi(\varepsilon^{n+1}_r) &= a\phi^n_r - \phi^n_r a = 0 \;\;\text{ and }\\
d^*_{n+1}\phi(\varepsilon^{n+1}_{r+1}) &=   b\phi^n_{r} + \phi^n_{r} b=0
\end{align*}
we obtain $\phi^n_{r}\in span_k\{a,ab\}$.

\textbf{For any other }$q\neq\pm 1$\\
If $n$  is even and $r$ is even, we obtain
\begin{align*}
d^*_{n+1}\phi(\varepsilon^{n+1}_r) &= a\phi^n_r - q^r \phi^n_r a = 0 \;\;\text{ and }\\
d^*_{n+1}\phi(\varepsilon^{n+1}_{r+1}) &=  q^{n-r} b\phi^n_{r} - \phi^n_{r} b=0
\end{align*}
 we obtain the trivial solutions $\phi^n_{r}\in span_k\{ab\}$.

If $n$  is even and $r$ is odd, we obtain
\begin{align*}
d^*_{n+1}\phi(\varepsilon^{n+1}_r) &= a\phi^n_r + q^r \phi^n_r a = 0 \;\;\text{ and }\\
d^*_{n+1}\phi(\varepsilon^{n+1}_{r+1}) &=  -q^{n-r} b\phi^n_{r} - \phi^n_{r} b=0
\end{align*}
with the trivial solutions $\phi^n_{r}\in span_k\{ab\}$.

In similar ways we obtain the trivial solutions for $n$ odd and $r$ even or odd. But we must have $\phi^n_r\in e_1\Lambda e_1, \;(0\leq r\leq n)$ and $\phi^n_{n+1}\in e_1\Lambda e_2$ , we come to conclude that $\phi^n_{r}\in span_k\{ab\}$, whenever $q\neq\pm 1$.\\ 
\\
The following table is a summary of solutions stating generators for each $\phi^n_r$;\\
\begin{tabular}{|p{1in}|p{1in}|p{1in}||p{1in}|p{1in}|}\hline
\textbf{$q= 1$} \\\hline
&\multicolumn{2}{|c|}{$n$ is even}&\multicolumn{2}{|c|}{$n$ is odd}\\\hline
&$r$ is even&$r$ is odd&$r$ is even&$r$ is odd\\\cline{1-5}
$\phi^n_0$&\multicolumn{2}{|c|}{$a,ab$}&\multicolumn{2}{|c|}{$a,ab$}\\\hline
$\phi^n_r$&$a,b,ab,e_1$&$ab$&$a,ab$&$b,ab$\\\cline{1-5}
$\phi^n_{n+1}$&\multicolumn{2}{|c|}{$c,bc$}&\multicolumn{2}{|c|}{$c,bc$}\\\hline
\end{tabular}\ \\
\begin{tabular}{|p{1in}|p{1in}|p{1in}||p{1in}|p{1in}|}\hline
\textbf{$q= -1$} \\\hline
&\multicolumn{2}{|c|}{$n$ is even}&\multicolumn{2}{|c|}{$n$ is odd}\\\hline
&$r$ is even&$r$ is odd&$r$ is even&$r$ is odd\\\cline{1-5}
$\phi^n_0$&\multicolumn{2}{|c|}{$a,ab$}&\multicolumn{2}{|c|}{$a,ab$}\\\hline
$\phi^n_r$&$ab,e_1$&$b,ab$&$a,b,ab$&$a,ab$\\\cline{1-5}
$\phi^n_{n+1}$&\multicolumn{2}{|c|}{$c,bc$}&\multicolumn{2}{|c|}{$c,bc$}\\\hline
\end{tabular}\ \\
\\
\begin{tabular}{|p{1in}|p{1in}|p{1in}||p{1in}|p{1in}|}\hline
\textbf{$q\neq \pm 1$} \\\hline
&\multicolumn{2}{|c|}{$n$ is even}&\multicolumn{2}{|c|}{$n$ is odd}\\\hline
&$r$ is even&$r$ is odd&$r$ is even&$r$ is odd\\\cline{1-5}
$\phi^n_0$&\multicolumn{2}{|c|}{$a,ab$}&\multicolumn{2}{|c|}{$a,ab$}\\\hline
$\phi^n_r$&$ab$&$ab$&$ab$&$ab$\\\cline{1-5}
$\phi^n_{n+1}$&\multicolumn{2}{|c|}{$c,bc$}&\multicolumn{2}{|c|}{$c,bc$}\\\hline
\end{tabular}\ \\
From these tables, we make the following deductions;
$$(n \text{ is even and } q=+1):\qquad dim(Ker d^*_{n+1}) = 2 + (\stackrel{(odd-positions)}{\frac{n}{2}\times 1} + \stackrel{(even-positions)}{\frac{n}{2}\times 4})+2 = 5(\frac{n}{2})+4$$
$$(n \text{ is odd and } q=+1):\qquad dim(Ker d^*_{n+1}) = 2 + (\stackrel{(odd-positions)}{\frac{n}{2}\times 2} + \stackrel{(even-positions)}{\frac{n}{2}\times 2})+2 = 2(n+2)$$
$$(n \text{ is even and } q=-1):\qquad dim(Ker d^*_{n+1}) = 2 + (\stackrel{(odd-positions)}{\frac{n}{2}\times 2} + \stackrel{(even-positions)}{\frac{n}{2}\times 2})+2 = 2(n+2)$$
$$(n \text{ is odd and } q=-1):\qquad dim(Ker d^*_{n+1}) = 2 + (\stackrel{(odd-positions)}{\frac{n}{2}\times 2} + \stackrel{(even-positions)}{\frac{n}{2}\times 3})+2 = 5(\frac{n}{2})+4$$
$$(\text{ for any }n,\; q\neq\pm 1):\qquad dim(Ker d^*_{n+1}) = 2 + (\stackrel{(odd-positions)}{\frac{n}{2}\times 1} + \stackrel{(even-positions)}{\frac{n}{2}\times 1})+2 = n+4$$
\end{proof}
\begin{rema}\label{bigrem0}
We note that these dimensions grow linearly as the homological dimension $n$ grows. We also observe from the tables in the proof of Proposition~\ref{prop} that if $q=\pm 1$, there are Hochschild n-cocycles of the form $\phi = (0,\cdots,0,e_1,0,\cdots,0)$ i.e $\phi^n_i = e_1$, where both $n$ and $i$ are even. We will later see that these are the only non-nilpotent elements. Whenever $n$ is odd, there is no $\phi$ for which $\phi^n_i = e_1$. This is equivalent to saying that all elements of odd homological degrees are nilpotent with respect to the cup product.
\end{rema}\ \\
\textbf{A Cup product formula on Hochschild cohomology}

We will now define a cup product formula at the chain level for cocycles in Hochschild cohomology. It was shown in ~\cite{ROKA} that the minimal projective resolution $\K$ of $\Lambda$ posseses a ``comultiplicative structure\rq\rq{}. That is, there is a comultiplication map $\Delta_{\K}: \K\rightarrow \mathbb{K}\otimes_{\Lambda}\mathbb{K}$ such that $(\Delta\ot1)\Delta = (1\ot\Delta)\Delta$. Recall that $\K$ embeds nicely into the bar complex $\B$ via $\iota:\K\rightarrow\B$. There is a map $\iota\ot\iota:\K\ot\K\rightarrow\B\ot\B$ such that $(\iota\ot\iota)(\K\ot\K) = \iota(\K)\ot\iota(\K)\subseteq \B\ot \B$. The following diagram is commutative.
\begin{align*}
&\mathbb{K} \stackrel{\Delta_{\mathbb{K}}}\longrightarrow \mathbb{K}\otimes_{\Lambda}\mathbb{K}\\
&\iota\downarrow \qquad\;\;\downarrow\iota\otimes\iota\\
&\mathbb{B} \stackrel{\Delta_{\mathbb{B}}}\longrightarrow \mathbb{B}\otimes_{\Lambda}\mathbb{B}
\end{align*}
that is $(\iota\otimes\iota)\circ\Delta_{\mathbb{K}} = \Delta_{\mathbb{B}}\circ\iota.$ We recall that if $\alpha\in \HHom_{\Lambda^e}(\mathbb{B}_m,\Lambda)\cong \HHom_k(\Lambda^{\otimes m}, \Lambda) $, and $\beta\in \HHom_k(\Lambda^{\otimes n}, \Lambda)$ are two cocycles, then one way to define the cup product is the composition of the following maps;
$$\alpha\smallsmile \beta: \mathbb{B}\stackrel{\Delta}\longrightarrow
\mathbb{B}\otimes_{\Lambda}\mathbb{B}\stackrel{\alpha\otimes \beta}\longrightarrow \Lambda\otimes_{\Lambda}\Lambda \stackrel{\pi}\simeq \Lambda$$
where $(\alpha\otimes \beta)(x\otimes y) = (-1)^{|\beta||x|}\alpha(x)\otimes \beta(y),$ and we take $|\beta|=n$ since it is an $n$-cocycle. This convention matches the traditional definition of the cup product on the bar resolution given as
$$(\alpha\smallsmile\beta)(a_1\otimes\cdots \otimes a_{m+n}) = (-1)^{mn} \alpha (a_1\otimes\cdots \otimes a_{m})\cdot \beta(a_{m+1}\otimes\cdots \otimes a_{m+n}), \qquad \forall \;a_i\in\Lambda$$
This definition extends to the resolution $\K$ using a similar composition map but replacing the $\Delta_{\mathbb{B}}$, that is
$$\phi\smallsmile \mu: \mathbb{K}\stackrel{\Delta_{\K}}\longrightarrow
\mathbb{K}\otimes_{\Lambda}\mathbb{K}\stackrel{\phi\ot\mu}\longrightarrow \Lambda\otimes_{\Lambda}\Lambda \stackrel{\pi}\simeq \Lambda$$
so that $\phi\smallsmile \mu = \pi(\phi\ot\mu)\Delta_{\K}$. Let $\phi: \K_m\rightarrow \Lambda$, and $\mu: \K_n\rightarrow \Lambda$ be two cocycles of homological degrees $m$ and $n$ respectively, we use the following notation $\phi\smallsmile\mu = (\phi^m_0,\phi^m_1,\cdots,\phi^m_m,\phi^m_{m+1})\smallsmile(\mu^n_0,\mu^n_1,\mu^n_2,\cdots,\mu^n_n,\mu^n_{n+1}) = ((\phi\mu)^{m+n}_0,(\phi\mu)^{m+n}_1,(\phi\mu)^{m+n}_2,\cdots,(\phi\mu)^{m+n}_{m+n},(\phi\mu)^{m+n}_{m+n+1})$ for their cup product.

\begin{prop}\label{prop0}
Let $\phi: \K_m\rightarrow \Lambda$, and $\mu: \K_n\rightarrow \Lambda$, be two Hochschild cocycles. Then the following gives
a formula for their cup product.\\
\\
$(\phi\smallsmile\mu)(\varepsilon^{m+n}_k) = (\phi\mu)^{m+n}_k = \begin{cases}(-1)^{mn} \phi^m_0\mu^n_0, & when\; k=0\\
(-1)^{mn} T^{m+n}_k & when\; 0<k<m+n\\
 (-1)^{mn}\phi^m_m\mu^n_n, & when\; k=m+n\\
 (-1)^{mn}\phi^m_0\mu^n_{n+1}, & when\;k=m+n+1
\end{cases}$
$$T^{m+n}_k = \sum_{j=max\{0,k-n\}}^{min\{m,k\}} (-q)^{j(n-k+j)}\phi^m_j\mu^n_{k-j}, \quad\qquad 0<k<m+n.$$
\end{prop}
\begin{proof}
Suppose $m=n=1$, take $\phi = (\phi^1_0,\phi^1_1,\phi^1_2) = (a,b,c) = (f^1_0, f^1_1, f^1_2)$ and $\mu = (\mu^1_0,\mu^1_1,\mu^1_2)=(a,b,c).$ We then realize that
$\phi\smallsmile\mu = (\phi^1_0,\phi^1_1,\phi^1_2)\smallsmile(\mu^1_0,\mu^1_1,\mu^1_2) = ( \phi^1_0\mu^1_0,\phi^1_0\mu^1_1 - q\phi^1_1\mu^1_0,\phi^1_1\mu^1_1 ,\phi^1_0\mu^1_2) = 
(a^2, ab-qba, b^2, ac) = (f^2_0, f^2_1, f^2_2, f^2_3)$ after applying $\pi(\phi\otimes \mu)\Delta_{\K}$. This consists of paths and linear combination of paths of length 2, that is, elements of $\Gamma^2$ given in~\eqref{gammas}. Similarly, if we take $m=1, n=2$, that is, take $\phi = (\phi^1_0,\phi^1_1,\phi^1_2) = (a,b,c)$ and $\mu = (\mu^2_0,\mu^2_1,\mu^2_2, \mu^2_3) = (a^2, ab-qba, b^2, ac)$. We obtain after applying  $\pi(\phi\otimes \mu)\Delta_{\K}$, 
$\phi\smallsmile\mu = (\phi^1_0\mu^2_0,\phi^1_0\mu^2_1 +(-q)^2 \phi^1_1\mu^2_0,\phi^1_0\mu^2_2 +(-q)^1 \phi^1_1\mu^2_1,\phi^1_1\mu^2_2 ,\phi^1_0\mu^2_3)
=(a^3, a^2b-qaba +q^2ba^2, ab^2-qbab+q^2b^2a, b^3, a^2c) = ( f^3_0, f^3_1, f^3_2, f^3_3, f^3_4) $ which are elements of  $\Gamma^3$ in ~\eqref{gammas}. It is enough to find a general formula for each element $f^{m+n}_r\in \Gamma^{m+n}$, so that $\iota(\varepsilon^{m+n}_r) = 1\ot f^{m+n}_r\ot 1.$ We will then find the image of $\varepsilon^{m+n}_r$ under $\Delta_{\K}$. Since $\phi\smallsmile\mu: \K_{m+n}\rightarrow \Lambda$, we simply evaluate
\begin{align*}
(\phi\smallsmile\mu)(\varepsilon^{m+n}_r) &= \pi (\phi\otimes \mu)\Delta_{\K} (\varepsilon^{m+n}_r), \;\text{ for } 0\leq r\leq m+n+1.
\end{align*}
Using the fact that 
\begin{align*}
 (\iota\otimes \iota) \Delta_{\K}(\varepsilon^{m+n}_{0}) &= \Delta_{\B}\iota (\varepsilon^{m+n}_{0}) \\
 &= \Delta_{\B}(1\otimes f^{m+n}_{0}\otimes 1) =  \stackrel{m+n\; times}{\Delta_{\B}(1\otimes f^1_0\otimes f^1_0\otimes \cdots \otimes f^1_0\otimes 1)}   \\
 & = \sum_{r=0}^{m+n} (1\otimes f^{r}_{0}\otimes 1)\otimes (1\otimes f^{m+n-r}_{0}\otimes 1)\\
 & = (\iota\otimes \iota)(\sum_{r=0}^{m+n} \varepsilon^{r}_{0}\otimes \varepsilon^{m+n-r}_{0}), \;\text{so that}\\
  \Delta_{\K}(\varepsilon^{m+n}_{0}) & = (\sum_{r=0}^{m+n} \varepsilon^{r}_{0}\otimes \varepsilon^{m+n-r}_{0}),\;\;\text{similarly}
  \end{align*}
  \begin{align*}
 \Delta_{\K}(\varepsilon^{m+n}_{m+n}) & = \sum_{r=0}^{m+n} \varepsilon^{r}_{r}\otimes \varepsilon^{m+n-r}_{m+n-r}\;\text{ and} \\
    \Delta_{\K}(\varepsilon^{m+n}_{m+n+1}) & = \sum_{r=0}^{m+n} \varepsilon^{r}_{0}\otimes \varepsilon^{m+n-r}_{m+n-r+1}
\end{align*}
We therefore obtain for $r=0, m+n, m+n+1$;\\
 $\pi (\phi\otimes \mu)\Delta_{\K} (\varepsilon^{m+n}_0) = \pi (\phi\otimes \mu) (\varepsilon^{m}_{0}\otimes \varepsilon^{n}_{0}) = (-1)^{mn}\pi (\phi^m_0\otimes \mu^n_0) = (-1)^{mn} \phi^m_0 \mu^n_0$.\\
$\pi (\phi\otimes \mu)\Delta_{\K} (\varepsilon^{m+n}_{m+n}) = \pi (\phi\otimes \mu) (\varepsilon^{m}_{m}\otimes \varepsilon^{n}_{n}) = (-1)^{mn}\pi (\phi^m_m\otimes \mu^n_n) =  (-1)^{mn}\phi^m_m \mu^n_n$.\\
$\pi (\phi\otimes \mu)\Delta_{\K} (\varepsilon^{m+n}_{m+n+1}) =  \pi (\phi\otimes \mu) (\varepsilon^{m}_{0}\otimes \varepsilon^{n}_{n+1}) =(-1)^{mn} \pi (\phi^m_0\otimes \mu^n_{n+1}) = (-1)^{mn} \phi^m_0 \mu^n_{n+1}$.\\
It was shown in ~\cite{FHFG} that for $r=1,2,\cdots,n-1,$
$$f^n_r = \sum_{j=max\{0,r+t-n\}}^{min\{t,r\}} (-q)^{j(n-r+j-t)} f^t_j\ot f^{n-t}_{r-j},$$
therefore
\begin{align*}
\iota(\varepsilon^{m+n}_r) &= 1\ot \Big{[}\sum_{j=max\{0,r+t-m-n\}}^{min\{t,r\}} (-q)^{j(m+n-r+j-t)} f^t_j\ot f^{m+n-t}_{r-j}\Big{]} \ot 1 \\
&\text{letting  } t=m\\
& = \sum_{j=max\{0,r-n\}}^{min\{m,r\}} (-q)^{j(n-r+j)} 1\ot f^m_j\ot f^{n}_{r-j}\ot 1 \\
\text{applying the diagonal map }\Delta_{\B} &\text{ and retaining the part that is nonzero when you apply }\phi\ot\mu
\end{align*}
\begin{align*}
(\Delta_{\B}\iota)(\varepsilon^{m+n}_r) &= \cdots + \sum_{j=max\{0,r-n\}}^{min\{m,r\}} (-q)^{j(n-r+j)} (1\ot f^m_j\ot 1) \ot (1\ot f^{n}_{r-j}\ot 1) +\cdots \\
& = \cdots +\sum_{j=max\{0,r-n\}}^{min\{m,r\}} (-q)^{j(n-r+j)}(\iota\ot\iota)( \varepsilon^m_j \ot \varepsilon^n_{r-j}) +\cdots \\
& \text{ using the relation that } (\iota\otimes\iota)\circ\Delta_{\mathbb{K}} = \Delta_{\B}\circ\iota \\
(\iota\ot\iota)\Delta_{\K}(\varepsilon^{m+n}_r) & = (\iota\ot\iota)\bigg[\cdots +\sum_{j=max\{0,r-n\}}^{min\{m,r\}} (-q)^{j(n-r+j)}( \varepsilon^m_j \ot \varepsilon^n_{r-j}) +\cdots \bigg]
\end{align*}
\begin{align*}
\Delta_{\mathbb{K}}(\varepsilon^{m+n}_r) &=\cdots +\sum_{j=max\{0,r-n\}}^{min\{m,r\}} (-q)^{j(n-r+j)} \varepsilon^m_j \ot \varepsilon^n_{r-j} +\cdots \\
&\text{ therefore  we obtain after applying } \phi\ot\mu\ \text{ and multiplication }\pi \\
(\phi\smallsmile\mu)(\varepsilon^{m+n}_{r}) &= (-1)^{mn}\sum_{j=max\{0,r-n\}}^{min\{m,r\}} (-q)^{j(n-r+j)} \phi(\varepsilon^m_j) \mu(\varepsilon^n_{r-j}) \\
&=(-1)^{mn}\sum_{j=max\{0,r-n\}}^{min\{m,r\}} (-q)^{j(n-r+j)} \phi^m_j\mu^n_{r-j} \\
&= (-1)^{mn}T^{m+n}_r 
\end{align*}
\end{proof}
\begin{rema}\label{bigrem1}
We give the following as a support to our previous Remark~\ref{bigrem0}. From Proposition~\ref{prop}, we observed that $ker d^*_{n+1}$ is generated by $\phi$ such that $\phi(\varepsilon^m_i) = \phi^m_i\in span_k\{a,b,ab,c,bc,e_1\}$.
But any $\phi$ having any of its $\phi^m_i$ to be any of $a,b,ab,c,bc$ is nilpotent. This is because
\begin{align}
(\phi\smallsmile\phi)(\varepsilon^{m+n}_{r})
&=(-1)^{mn}\sum_{j=max\{0,r-n\}}^{min\{m,r\}} (-q)^{j(n-r+j)} \phi^m_j\phi^n_{r-j} 
\end{align}
where $\phi^m_j\phi^n_{r-j}$ is a product of any two elements in the set $\{a,b,ab,c,bc\}$ which is equal to 0 in the algebra. If it is not zero,
we simply take a triple cup product using the following;
\begin{align*}
(\phi\smallsmile\phi\smallsmile\phi)(\varepsilon^{n+n+n}_{r}) &= (\mu\smallsmile\phi)(\varepsilon^{m+n}_{r}) \quad(\text{take  }\mu = \phi\smallsmile\phi, m=n+n)\\
&=(-1)^{mn}\sum_{j=max\{0,r-n\}}^{min\{m,r\}} (-q)^{j(n-r+j)} \mu(\varepsilon^m_j)\phi(\varepsilon^n_{r-j}) \\
&=(-1)^{mn}\sum_{j=max\{0,r-n\}}^{min\{m,r\}} (-q)^{j(n-r+j)} [\phi\smallsmile\phi(\varepsilon^{n+n}_j)]\phi(\varepsilon^n_{r-j}) \\
&=(-1)^{mn}\sum_{j=max\{0,r-n\}}^{min\{m,r\}} (-q)^{j(n-r+j)} \Big[(-1)^{n^2}\sum_{i=max\{0,l-n\}}^{min\{n,l\}} (-q)^{i(n-l+i)} \phi(\varepsilon^n_i)\phi(\varepsilon^n_{l-i}) \Big]\phi(\varepsilon^n_{r-j}) \\
&=(-1)^{3n^2}\sum_{j=max\{0,r-n\}}^{min\{m,r\}} \sum_{i=max\{0,l-n\}}^{min\{n,l\}}(-q)^{ij(n-r+j)(n-l+i)} \phi(\varepsilon^n_i)\phi(\varepsilon^n_{l-i})\phi(\varepsilon^n_{r-j}) 
\end{align*}
The product $\phi(\varepsilon^n_i)\phi(\varepsilon^n_{l-i})\phi(\varepsilon^n_{r-j}) = \phi^n_i\phi^n_{l-i}\phi^n_{r-j}$ is always 0 in $\Lambda_q$ except each $\phi^n_i=e_1.$ Therefore a cocycle $\phi\in HH^m(\Lambda)$ is non-nilpotent if and only if
$\phi^m_i =\phi^m_{l-i} =\phi^m_{r-j} = e_1$ for some $i,j,l,r$. According to Proposition~\ref{prop} this is the case only when $q=\pm 1, n$ is even and $i$ is even.
\end{rema}\ 
The following Proposition summarizes Remarks~\ref{bigrem0} and  ~\ref{bigrem1}.
\begin{prop}\label{prop1}
Let $\phi:\K_n\rightarrow \Lambda_q,$ be a cocycle. Then $\phi$ is non-nilpotent if, and only if $q=\pm 1$ and
$$\phi(\varepsilon^{n}_r) = \phi^{n}_r = \begin{cases} e_1, &\text{if $n,\;r$ are even} \\ 0 & \text{ otherwise. }\end{cases}$$
\end{prop}\ \\
Let $C^n(\Lambda,\Lambda) = \HHom_{\Lambda^e}(\K_n,\Lambda)$ be the $\Lambda^e$-module generated by all $n$-Hochschild cochains. Denote by $C^{*}(\Lambda,\Lambda) = \bigoplus_{n\geq 0} \HHom_{\Lambda^e}(\K_n,\Lambda)$ the algebra of Hochschild cochains with coefficients in $\Lambda$. The $n$-cocycles of $\HH^*(\Lambda_{\pm 1})$ given in Proposition~\ref{prop1} are therefore given by $Z^{*}(\Lambda,\Lambda) = C^{*}(\Lambda,\Lambda)/\mathcal{N}.$

Before we state the next proposition, we will show that no two cocycles of Proposition~\ref{prop1} differ by a coboundary. This is important as we will later define a 1-1 module homomorphism from $Z^{*}(\Lambda,\Lambda)$ to the polynomial ring $k[x,y]$.

For a fixed $n$, let $\phi,\beta$ be two cocycles such that $\phi(\varepsilon^{2n}_r) = \phi^{2n}_r = e_1$, $\beta(\varepsilon^{2n}_s) = \beta^{2n}_s = e_1$, where $r<s$ are both even and there is $\alpha$ such that $d^{*}(\alpha) = \phi-\beta =(0,\cdots,0,e_1,0,\cdots,0,e_1,0,\cdots,0),$ the idempotents $e_1$ is in the $r$ and $s$ position. We have the following when $i=0$,
\begin{align*}
d^*(\alpha)(\varepsilon^{2n}_{i}) &= \alpha d(\varepsilon^{2n}_{i}), \\
 0 &= \alpha(a\varepsilon^{2n-1}_0 + (-1)^n\varepsilon^{2n-1}_0a) =  a\alpha(\varepsilon^{2n-1}_0) + (-1)^2n\alpha(\varepsilon^{2n-1}_0)a\\
\text{hence,  }\quad  &\alpha(\varepsilon^{2n-1}_0) = 0
\end{align*}
and in general, we must have 
\begin{align*}
e_1 = d^*(\alpha)(\varepsilon^{2n}_{r}) &=  a\alpha(\varepsilon^{2n-1}_r) + (-1)^{2n-r}q^r \alpha(\varepsilon^{2n-1}_r)a + 
(-q)^{2n-r}b\alpha(\varepsilon^{2n-1}_{r-1}) + (-1)^{2n} \alpha(\varepsilon^{2n-1}_{r-1})b\\
\text{ and }\\
e_1 = d^*(\alpha)(\varepsilon^{2n}_{s}) &=  a\alpha(\varepsilon^{2n-1}_s) + (-1)^{2n-s}q^s \alpha(\varepsilon^{2n-1}_s)a + 
(-q)^{2n-s}b\alpha(\varepsilon^{2n-1}_{s-1}) + (-1)^{2n} \alpha(\varepsilon^{2n-1}_{s-1})b.
\end{align*}

There is no way we can define $\alpha(\varepsilon^{2n-1}_r)$ and  $\alpha(\varepsilon^{2n-1}_s)$ so that the equation above is true i.e., the right hand side equals $e_1$. Hence there is no such $\alpha$. Therefore each cocycle is distinct and do not differ by a coboundary. We now make the following Proposition. 
\begin{prop}\label{prop2}
$Z^{*}(\Lambda,\Lambda)$ is graded with respect to the cup product and can be expressed as a subalgebra of $k[x,y]$, that is
$$Z^{*}(\Lambda,\Lambda) \cong k[x^2,y^2]y^2$$
where the degree of $y$ is 1, and $xy$ is 2.
\end{prop}
\begin{proof}
We first show that $Z^{*}(\Lambda,\Lambda)$ can be expressed as \\ $\bigoplus_{n>0} span_k\Big\{ \phi:\K_{2n}\rightarrow\Lambda_{\pm 1} \Big{|}  \phi^{n}_r = \begin{cases} e_1, &\text{if $r$ is even} \\ 0 & \text{ otherwise }\end{cases}\Big\}.$ It is straightforward to see that if $\phi\in Z^{*}(\Lambda,\Lambda)$, there are pairs of positive even integers $m,i$ such that $\phi: \K_{m}\rightarrow \Lambda_{\pm 1}$ and $\phi(\varepsilon^m_i) = e_1$ and $0$ at other positions not equal to $i$. Hence 
\begin{align*}
Z^{*}(\Lambda,\Lambda)&\subseteq span_k\Big\{ \phi:\K_{m}\rightarrow\Lambda_{\pm 1} \Big{|}  \phi^{m}_i = \begin{cases} e_1, &\text{if $i$ is even} \\ 0 & \text{ otherwise }\end{cases}\Big\}\\
&\subseteq \bigoplus_{n>0} span_k\Big\{ \phi:\K_{2n}\rightarrow\Lambda_{\pm 1} \Big{|}  \phi^{n}_r = \begin{cases} e_1, &\text{if $r$ is even} \\ 0 & \text{ otherwise }\end{cases}\Big\}
\end{align*}
The grading comes from the fact that its elements are Hochschild cocycles and all odd degree elements vanish. We will now show that $\bigoplus_{n>0} span_k\big\{ \phi:\K_{2n}\rightarrow\Lambda_{\pm 1} \Big{|}  \phi^{n}_r = e_1, \text{for some $r$}\big\} $
is graded with respect to the cup product, hence contained in $Z^{*}(\Lambda,\Lambda)$.\\
Let $\phi:\K_{2m}\rightarrow \Lambda_q,$ be given by 
$ \phi^{2m}_r = \begin{cases} e_1, &\text{if $r$ is even} \\ 0 & \text{ otherwise }\end{cases},$ and \\
$\mu:\K_{2n}\rightarrow \Lambda_q,$ be given by 
\( \mu^{2n}_s = \begin{cases} e_1, &\text{if $s$ is even} \\ 0 & \text{ otherwise }\end{cases}.\) \\
First note that $(\phi\smallsmile\mu)(\varepsilon^{2(m+n)}_{0}) =  \phi^{2m}_{0}\mu^{2n}_{0} = 0$. Since $2m+2n$ is even, whenever $r=2m$ and $s=2n$, we get $(\phi\smallsmile\mu)(\varepsilon^{2(m+n)}_{2(m+n)}) =  \phi^{2m}_{2m}\mu^{2n}_{2s} = e_1\cdot e_1 = e_1.$ Also $(\phi\smallsmile\mu)(\varepsilon^{2(m+n)}_{2(m+n)+1}) =  \phi^{2m}_{0}\mu^{2n}_{2n+1} = 0.$ Whenever $0\leq t\leq 2(m+n)$
\begin{align*}
(\phi\smallsmile\mu)(\varepsilon^{2m+2n}_{t})
&=\sum_{j=max\{0,t-2n\}}^{min\{2m,t\}} (-q)^{j(2n-t+j)} \phi^{2m}_j\mu^{2n}_{t-j} = \pm  e_1 
\end{align*}
whenever $ t,j$ are even since $\phi^{2m}_j\mu^{2n}_{t-j} =  \begin{cases} e_1, &\text{if $t,j$ is even} \\ 0 & \text{ otherwise. }\end{cases}$ \ \\
This shows that  $(\phi\smallsmile\mu)\in span_k\Big\{ \phi:\K_{2n}\rightarrow\Lambda_q \Big{|}  \phi^{n}_r = \begin{cases} e_1, &\text{if $r$ is even} \\ 0 & \text{ otherwise }\end{cases}\Big\}.$ We define a map from $Z^{*}(\Lambda,\Lambda)\rightarrow k[x,y]$ by\\
$ (0,0,e_1,0,\cdots,0) \mapsto x^{2(n-1)}y^2,\quad (0,0,0,0,e_1,0,\cdots,0) \mapsto x^{2(n-2)}y^4,\\
\cdots, \underbrace{(0,\cdots, 0, e_1 , 0,\cdots,0)}_{r-th\;position} \mapsto x^{2n-r}y^r,\cdots,
(0,0,\cdots,0,e_1,0)\mapsto y^{2n}$. Under this map, the image of $Z^{*}(\Lambda,\Lambda)$ is the subalgebra $k[x^2,y^2]y^2$ which is not finitely generated as an algebra. Also note how the cup product corresponds with multiplication in $k[x,y]$, that is
\begin{equation*}
\xymatrix{ 
 (0,\cdots, 0, \underbrace{e_1}_{r} , 0,\cdots,0) \smallsmile(0,\cdots, 0, \underbrace{e_1}_{s} , 0,\cdots,0)  \ar@{-->}[d]_-{=} \,\,\ar[r]^-{ = } & (x^{2n-r}y^r)\cdot(x^{2m-s}y^s) \ar@{-->}[d]_{=} \\
 (0,\cdots, 0, \underbrace{e_1}_{r+s} , 0,\cdots,0) \ar[r]^-{ = } & x^{2(n+m)-(r+s)}y^{r+s} }
\end{equation*}

For each $n$, the element $x^{2(n-1)}y^2$ which we identify with $(0,0,e_1,0,\cdots,0)$ cannot be generated by any element of lower homological degree.  This brings us to conclude that the graded copies of cocycles of Proposition ~\ref{prop1} can be compactly given as \\
$\bigoplus_{n>0} span_k\big\{ \phi:\K_{2n}\rightarrow\Lambda_{\pm 1} \Big{|}  \phi^{n}_r = e_1, \text{for some $r$}\big\} = k[x^2,y^2]y^2$. 
\end{proof}\ \\
Before we give the final Theorem, we illustrate with an example.\\
\textbf{Example}
To show that
\begin{align*}
x^2y^2\cdot y^2 &= (0,0,e_1,0,0,0)\smallsmile(0,0,e_1,0)\\
&=(0,0,0,0,e_1,0,0,0) = x^2\cdot y^4
\end{align*}
Take $\phi = x^2y^2= (\phi^4_0,\phi^4_1,\phi^4_2,\phi^4_3,\phi^4_4,\phi^4_5) = (0,0,e_1,0,0,0)$ and 
$\mu =y^2= (\phi^2_0,\phi^2_1,\phi^2_2,\phi^2_3) = (0,0,e_1,0)$
\begin{align*}
(\phi\smallsmile\mu)(\varepsilon^6_0) &= \phi^4_0\mu^2_0 = 0 \\
(\phi\smallsmile\mu)(\varepsilon^6_1) &=\sum_{j=0}^{1} (-1)^{j(1+j)} \phi^4_j\mu^2_{1-j} = \phi^4_0\mu^2_1 + \phi^4_1\mu^2_0 = 0\\  
(\phi\smallsmile\mu)(\varepsilon^6_2) &=\sum_{j=0}^{2} (-1)^{j^2} \phi^4_j\mu^2_{2-j} = \phi^4_0\mu^2_2 - \phi^4_1\mu^2_1 + \phi^4_2\mu^2_0 = 0\\
(\phi\smallsmile\mu)(\varepsilon^6_3) &=\sum_{j=1}^{3} (-1)^{j(-1+j)} \phi^4_j\mu^2_{3-j} = \phi^4_1\mu^2_2 + \phi^4_2\mu^2_1 + \phi^4_3\mu^2_0 = 0\\
(\phi\smallsmile\mu)(\varepsilon^6_4) &=\sum_{j=2}^{4} (-1)^{j(-2+j)} \phi^4_j\mu^2_{4-j} = \phi^4_2\mu^2_2 - \phi^4_3\mu^2_1 + \phi^4_4\phi^2_0 = e_1\\
(\phi\smallsmile\mu)(\varepsilon^6_5) &=\sum_{j=3}^{4} (-1)^{j(-3+j)} \phi^4_j\mu^2_{5-j} = \phi^4_3\mu^2_2 + \phi^4_4\mu^2_1 = 0\\
(\phi\smallsmile\mu)(\varepsilon^6_6) &= \phi^4_4\mu^2_4 = 0\\
(\phi\smallsmile\mu)(\varepsilon^6_7) &= \phi^4_0\mu^2_3 = 0
\end{align*}
\\
\begin{theo}\label{tolwit}
Let $k$ ($char(k)\neq 2$) be a field and $\Lambda_q = \frac{kQ}{I}$ be the family of quiver algebras of ~\eqref{quivalg}, and $\mathcal{N}$ the set of nilpotents elements of $\HH^*(\Lambda_q),$  then
$$\HH^*(\Lambda_q)/\mathcal{N} = \begin{cases} \HH^0(\Lambda_q)/\mathcal{N} \cong Z(\Lambda_q)^* \cong k,  &\text{ if } q\neq \pm 1 \\
Z(\Lambda_q)^* \oplus k[x^2,y^2]y^2 \cong k \oplus k[x^2,y^2]y^2 , &\text{ if } q =\pm 1 \end{cases}$$
where the degrees of $y$ is 1, and that of $xy$ is 2.
\end{theo}

\begin{proof}
If $q\neq\pm 1$, then all $\phi:\K_n\rightarrow \Lambda_q$ are nilpotent elements by Remark\ref{bigrem1} and Proposition~\ref{prop1}. From Remark~\ref{bigrem}, we have then that 
$$\HH^*(\Lambda_q)/\mathcal{N} =  \HH^0(\Lambda_q)/\mathcal{N} \cong Z(\Lambda_q)^* \cong k$$
If $q=\pm 1$, then the only non-nilpotent elements are those of Proposition~\ref{prop1}. From Remarks~\ref{bigrem} and Proposition ~\ref{prop2} we have that Hochschild cohomology ring modulo nilpotent elements of the family of quiver algebras in~\eqref{quivalg} is spanned by graded copies of cocycles given by Proposition~\ref{prop1}. That means that  
\begin{align*}
\HH^*(\Lambda_q)/\mathcal{N} &= \HH^0(\Lambda_{\pm 1})/\mathcal{N} \oplus Z^*(\Lambda,\Lambda)\\
&\cong k \oplus\Big( \bigoplus_{n>0} span_k\Big\{ \phi:\K_{2n}\rightarrow\Lambda_q \big{|}  \phi^{n}_r = \begin{cases} e_1, &\text{if $n,\;r$ is even} \\ 0 & \text{ otherwise }\end{cases}\Big\}\big)\\
&=  k \oplus k[x^2,y^2]y^2.
\end{align*}
\end{proof}

\textbf{Acknowledgment:} I want to thank Dr. Sarah Witherspoon for reading through the manuscript and making invaluable suggestions.


\newpage
\begin{thebibliography}{99}
\bibitem{FHFG}
\textsc{R.O. Buchweitz, E.L. Green, D. Madsenl, \O. Solberg,}
Hochshchild cohomology without finite global dimension,
\textit{Math. Res. Lett.,} \textbf{vol. 12 }(2005), no. 6, 805-816, Database: arXiv

\bibitem{MSKA}
\textsc{R.O. Buchweitz, E.L. Green, N. Snashall, \O. Solberg,}
Multiplicative structures for Koszul algebras,
\textit{The Quarterly Journal of Mathematics} \textbf{vol. 59(4)} (2008), 441-454 , Database: arXiv

\bibitem{GX}
\textsc{E. Gawell, Q. R. Xantcha},
Centers of partly (anti-)commutative quiver algebras and finite generation of the Hochschild cohomology rings,
\textit{ Manuscripta math.} (2016) 150: 383.

\bibitem{ROKA}
\textsc{E. L. Green, G. Hartman, E. N. Marcos, \O. Solberg,}
Resolutions over Koszul algebras,
\textit{Archiv der Mathematik}  textbf{vol. 85(2)} (2005) 118-127, arXiv:math/0409162

\bibitem{MV}
\textsc{R. Martinez-Villa},
Introduction to Koszul algebras,
\textit{Rev. Un. Mat. Argentina} \textbf{ 58 }(2008), no.2, 67-95.

\bibitem{SVHC}
\textsc{N. Snashall,}
Support varieties and the Hochschild cohomology ring modulo nilpotence,
\textit{Proceedings of the 41st Symposium on Ring Theory and Representation Theory, Tsukuba} (2009), pp. 68–82
 
\bibitem{SVHCR}
\textsc{N. Snashall, \O. Solberg,}
Support varieties and the Hochschild cohomology rings,
\textit{Proc. London Maths. Soc.} \textbf{ 88 }(2004), no. 3, 705–732

\bibitem{HCSW}
  \textsc{S. Witherspoon},
Hochschild Cohomology for Algebras,
\textit{Graduate Studies in Mathematics, American Mathematical Society,} \textbf{to appear}.

\bibitem{XU}
\textsc{F. Xu},
Hochschild and ordinary cohomology rings of small categories, 
\textit{Adv Math.} \textbf{219}  (2008), 1872-1893.

\bibitem{MCSS}
\textsc{Y. Xu, C. Zhang},
\textit{More counterexamples to Happel's question and Snashall-Solberg's conjecture}, arXiv:1109.3956

\end{thebibliography}
\end{document}